\documentclass[12pt,reqno]{amsart}
\usepackage{geometry}                
\usepackage{amsmath,amssymb}
\usepackage{graphicx}
\usepackage{amssymb,latexsym, amsmath}
\usepackage{amsfonts,amsthm}
\usepackage{amscd}
\usepackage{mathrsfs}
\usepackage{hyperref}
\usepackage{eucal}
\usepackage{epstopdf}
\usepackage{amsgen}
\usepackage{xspace}
\usepackage{verbatim}
\usepackage{stmaryrd}
\usepackage{enumitem}
\newlist{steps}{enumerate}{1}
\setlist[steps, 1]{label = Step \arabic*:}

\newcommand{\gd}{\Delta}
\newcommand{\de}{\delta}

\newcommand{\inpt}[1]{\langle #1 \rangle}

\newcommand{\gw}{\Omega}
\newcommand{\ap}{\alpha}
\newcommand{\ga}{\gamma}
\newcommand{\gb}{\beta}

\newcommand{\gl}{\lambda}

\newcommand{\ms}{\mathscr}

\newcommand{\nb}{\nabla}
\newcommand{\vp}{\varphi}
\newcommand{\ve}{\varepsilon}
\newcommand{\pdr}{\partial}

\newcommand{\csg}{\{ S(t)\}_{t\geq0}}

\newcommand{\beq}{\begin{equation}}
\newcommand{\eeq}{\end{equation}}
\newcommand{\bea}{\begin{align}}
\newcommand{\eea}{\end{align}}
\newcommand{\bthm}{\begin{theorem}}
\newcommand{\ethm}{\end{theorem}}
\newcommand{\bpr}{\begin{proof}}
\newcommand{\epr}{\end{proof}}
\newcommand{\bcl}{\begin{corollary}}
\newcommand{\ecl}{\end{corollary}}
\newcommand{\bpn}{\begin{proposition}}
\newcommand{\epn}{\end{proposition}}
\newcommand{\bre}{\begin{remark}}
\newcommand{\ere}{\end{remark}}
\newcommand{\bdf}{\begin{definition}}
\newcommand{\edf}{\end{definition}}
\newcommand{\bss}{\begin{align*}}
\newcommand{\ess}{\end{align*}}

\newcommand{\bl}{\label}

\newcommand{\dist}{\operatorname{dist}}

\newtheorem{theorem}{Theorem}[section]
\newtheorem{corollary}[theorem]{Corollary}

\newtheorem{lemma}[theorem]{Lemma}
\newtheorem{proposition}[theorem]{Proposition}

\theoremstyle{definition}
\newtheorem{definition}[theorem]{Definition}
\theoremstyle{remark}
\newtheorem{remark}{Remark}

\numberwithin{equation}{section}

\title{Diffusive Hindmarsh-Rose Equations with Memristors}

\begin{document}

\title [Diffusive Hindmarsh-Rose Equations with Memristors]{Global Dynamics of Diffusive Hindmarsh-Rose Equations with Memristors}

\author[Y. YOU]{YUNCHENG YOU} $^{\dag}$
\address{Professor Emeritus, University of South Florida, Tampa, FL 33620, USA}
\email{you@mail.usf.edu}

\thanks{}



\date{August 21, 2022}



\begin{abstract}
Global dynamics of the diffusive Hindmarsh-Rose equations with memristors as a new proposed model for neuron dynamics are investigated in this paper. We prove the existence and regularity of a global attractor for the solution semiflow through uniform analytic estimates showing the higher-order dissipative property and the asymptotically compact characteristics of the solutions by the approach of Kolmogorov-Riesz theorem. The quantitative bounds of the regions containing this global attractor respectively in the state space and in the regular space are explicitly expressed by the model parameters.
\end{abstract}

\maketitle

\section{\textbf{Introduction}}

Starting from the well-known Hodgkin-Huxley equations \cite{HH} (1952), which provided a highly nonlinear four-dimensional model for general neuron dynamics, and the two-dimensional FitzHugh-Nagumo equations \cite{FH} (1961-1962) as a simplified model which explains periodic firing with refractory but not able to generate chaotic neuron burstings, scientists have proposed various types of mathematical neuron models based on the biological characteristics of neuron functions and the biophysical laws. Two key issues in any modeling of neuron dynamics and neuronal networks are the firing-bursting patterns of single neurons and the collective behaviors of neural networks, especially synchronization and chaotic dynamics. All these issues are closely linked to significant applications in many areas such as brain diseases, image and signal processing, encryption of communications, and mostly artificial neural networks and artificial intelligence. 

The Hindmarsh-Rose equations \cite{HR} (1984) is originally a three-dimensional ODE model for neuron firing-bursting phenomena and has been studied through bifurcation analysis and numerical simulations by many researchers, cf. \cite{CK, CS, ET, HR, IG, Ng} and the references therein. This model exhibits rich and interesting spatial-temporal bursting patterns \cite{BRS, IG, SC, SPH}. In particular, the three-dimensional complex bifurcations lead to numerically observed  and sophisticated chaotic bursting behaviors. 

Very recently the author's group studied global dynamics generated by the spatially diffusive Hindmarsh-Rose equations \cite{PYS1, PY, PYS2}, random dynamics of the stochastic Hindmarsh-Rose equations \cite{CY}, and synchronization of complex Hindmarsh-Rose neural networks and FitzHugh-Nagumo neural networks \cite{CPY, PSY}.

In this work, we propose and study global dynamics of the diffusive Hindmarsh-Rose equations with memristors, which is a new mathematical model for neuron dynamics in terms of a hybrid system of PDE and ODE featuring an additional component equation for memristors and its nonlinear coupling to the membrane potential equation for a neuron cell. 

The concept memristor (meaning a memory resistor) was coined by Leon Chua \cite{Chua} (1971), as an electrical device with two terminals, which denotes the relationship between time-varying electromagnetic flux and electric charges. General memristive system was initially tackled in \cite{ChuaK} (1976) and has attracted broad scientific interests since the seminal paper \cite{SS} (2008) published in Nature.

The memristors are recognized and used in advanced neuron models to describe the electromagnetic induction effect caused by ions movement across the neuron cell membrane, which has been observed through fluctuations of extracellular calcium and potassium ions' concentrations in experiments \cite{QW, US2}. Moreover, the memristor synapsis in a model carries and transmits dynamically memorized information, which serves as a different type of synapses in neuron networks beside the electrical synapses and chemical synapses well known in neuroscience \cite{ET, US1, Wu}. 

The research results on memristive Hindmarsh-Rose neuron models and FitzHugh-Nagumo neuron models have been rapidly expanding in the recent decade, cf. \cite{Ay} - \cite{BRS}, \cite{DR, EE, Han}, \cite{QW} - \cite{RM}, \cite{SWZ, US1, US2}, \cite{Wu} - \cite{YP} and many references therein. Memristive neuron networks and artificial neural networks currently become an active topic as it shows rich collective dynamical behaviors and chaotic bursting patterns \cite{Han, US1} by varying the coupling strengths and other parameters, such as the reported coexisting chimeras and local attractors \cite{BB, Han, RM, SR} and clusters \cite{BB, Wang}, and synchronization enhanced by memristive couplings  \cite{EE, Guan, HY, RJ, SR, US2, VK, XJ}.   

In this paper, we consider the following new model of the diffusive Hindmarsh-Rose equations with memristor for a single neuron:
\begin{align}
    \frac{\pdr u}{\pdr t} & = \eta \gd u + a u^2 - b u^3 + v - w + J_e - k_1 \vp (\rho) u, \, \bl{ueq} \\
    \frac{\pdr v}{\pdr t} & = \ap - \gb u^2 - v, \bl{veq} \\
    \frac{\pdr w}{\pdr t} & = q (u - u_e) - rw, \bl{weq} \\
    \frac{\pdr \rho}{\pdr t} & = u - k_2 \rho, \bl{peq}
\end{align}
for $t > 0,\; x \in \gw \subset \mathbb{R}^{n}$ ($n \leq 3$), where $\gw$ is a bounded domain up to three dimension with locally Lipschitz continuous boundary (put in a general mathematical scope). The nonlinear term in a quadratic form
\beq \bl{pp}
	\vp (\rho) = c + \ga \rho + \de \rho^2, \quad c, \ga \in \mathbb{R}, \;\; \delta > 0,
\eeq 
presents the memristive coupling in the equation of membrane potential \eqref{ueq}, where the memristive variable $\rho (t, x)$ stands for the memductance of the memristor and $\vp (\rho)$ represents the electromagnetic induction flux with its coupling strength $k_1$ and self-coupling strength $k_2$ respectively. All the results proved in this paper are also valid for another type of memristor \cite{Wang}, $\vp (\rho) = \tanh (\rho)$, simply by adjusting the estimates in proof. 

In this system \eqref{ueq}-\eqref{peq}, the variable $u(t, x)$ refers to the membrane electric potential of a neuron cell, the variable $v(t, x)$ represents the transport rate of the ions of sodium and potassium through the fast channels and can be called the spiking variable, while the variables $w(t, x)$ represents the transport rate across the neuron membrane through slow channels of calcium and other ions correlated to the inter-spike quiescence and can be called the bursting variable. 

All the involved parameters $a, b, \eta, \ap, \gb, q, r,  \de, k_1, k_2$ and the external input $J_e$ can be any positive constants, while the reference value of the membrane potential $u_e \in \mathbb{R}$ and the first two parameters in \eqref{pp} $c, \ga \in \mathbb{R}$ can be any real number constants. For instance, a set of typical parameter values can be \cite{RM, SWZ, SPH, US1}
\begin{gather*}
	 J_e = 3.2, \;\; r = 0.002, \; \; q = 0.008,  \;\; u_e = -1.6,  \\[2pt]
	 a = 3.0, \;\; b = 1, \;\; \ap = 1.0, \;\; \gb = 5.0, \\[2pt]
	 \gamma = 0.4, \;\; \delta = 0.8, \;\; k_1 = 0.9, \;\; k_2 = 6.5.
\end{gather*}
We impose the homogeneous Neumann boundary conditions for the $u$-component,
\begin{equation} \label{nbc}
    \frac{\pdr u}{\pdr \nu} \, (t, x) = 0, \quad  t > 0,  \; x \in \partial \gw ,
\end{equation}
and the initial conditions of the components are denoted by
\begin{equation} \bl{inc}
   u_0 (x) = u(0, x), \;  v_0 (x) = v(0, x), \; w_0 (x) = w(0, x),  \; \rho_0 (x) = \rho (0, x), \quad x \in \gw.
\end{equation}

In the listed and many other references, the methodology of investigations into the memristive Hindmarsh-Rose neuron models mainly consists of bifurcation and stability analysis supported with numerical simulations to imitate neuron bursting-firing patterns. Several commonly used methods in this area are bifurcation diagrams and Lyapunov exponents \cite{Ay, BB, EE, RM, SR, US2}, generalized Hamiltonian functions and Lyapunov functions \cite{Ay, US1, US2, XQM}, center manifold theory \cite{Ay, RJ, Wang, XQM}, dissipativity analysis \cite{Wang}, algebraic invariant manifold for analytic solutions \cite{Ay}, etc.

Notably the proposed memristive neuron model of diffusive Hindmarsh-Rose equations in this paper reflects the structural features of a neuron cell that has the short-branch dendrites receiving incoming signals and the long-branch axon (naturally viewed as a one-dimensional space) propagating outreaching signals, which justifies the diffusive partial differential equation of the membrane potential in \eqref{ueq}. 
 
We shall present in Section 2 the formulation of the system \eqref{ueq}-\eqref{peq} and the preliminaries. In Section 3 we shall conduct uniform estimates to show the absorbing property of this solution semiflow. In Section 4 and Section 5  we shall prove the higher-order dissipativity  and the asymptotic compactness of the solution semiflow by means of the Kolmogorov-Riesz theorem. Finally in Section 6, the main result on the existence and regularity of a global attractor, which characterizes the collection of all permanent regimes of the modeled neuron dynamics, will be proved and the quantitative bounds of the regions containing this global attractor respectively in the state space and in the regular space are explicitly expressed by the model parameters.

\section{\textbf{Formulation and Preliminaries}}

For the diffusive Hindmarsh-Rosse equations with memristor \eqref{ueq} - \eqref{peq} proposed in this paper, we define the state space to be $E = [L^2 (\gw)]^4 = L^2 (\gw, \mathbb{R}^4)$, which is a Hilbert space and can be roughly called the \emph{energy space}. Also define the \emph{mild space} $\Pi = H^1 (\gw) \times L^2 (\gw, \mathbb{R}^3)$ and the \emph{regular space} $\Gamma = H^2 (\gw) \times L^\infty (\gw, \mathbb{R}^3)$, where $H^1 (\gw)$ and $H^2 (\gw)$ are the Sobolev spaces. \,The norm and inner-product of the Hilbert space $L^2 (\gw)$ or $E$ will be denoted by $\| \, \cdot \, \|$ and $\inpt{\,\cdot , \cdot\,}$, respectively. The norm of Banach space $L^p (\gw)$ will be dented by $\| \cdot \|_{L^p}$ if $p \neq 2$. We shall use $| \cdot |$ to denote either a vector norm or a set measure in a Euclidean space.

The initial-boundary value problem \eqref{ueq}--\eqref{inc} is usually formulated as an initial value problem of an evolutionary equation:
\begin{equation} \label{pb}
 	\begin{split}
   	& \frac{\partial g}{\partial t} = A g + f(g), \quad  t > 0, \\[2pt]
    	g(0) &= g_0 = \textup{col}\, (u_0, v_0, w_0, \rho_0) \, \in E.
	\end{split}
\end{equation}
Here the vector function $g(t, x) = \textup{col} \,(u(t, x), v(t, x), w(t, x), \rho (t, x))$, the nonpositive self-adjoint operator
\begin{equation} \label{opAh}
        A =
        \begin{pmatrix}
            \eta \gd  & 0   & 0 & 0   \\[3pt]
            0 & - I   & 0  & 0  \\[3pt]
            0 & 0 & - r I  & 0  \\[3pt]
            0 & 0 & 0 & - k_2 I
        \end{pmatrix}
        : \mathcal{D} (A) \rightarrow E,
\end{equation}
whose domain $\mathcal{D} (A) = \{g \in H^2(\gw) \times L^2 (\gw, \mathbb{R}^3): \pdr g /\pdr \nu = 0\; \textup{on the boundary} \, \pdr \gw\}$, is the generator of an analytic $C_0$-semigroup $\{e^{At}\}_{t \geq 0}$ on the Hilbert space $E$, cf. \cite{SY}. 
The nonlinear mapping in the equation \eqref{pb},
\begin{equation} \label{opfh}
        f(g) = 
        \begin{pmatrix}
             au^2 - bu^3 + v - w + J_e - k_1 \vp (\rho)u  \\[4pt]
             \ap - \gb u^2  \\[4pt]
	     q (u - u_e)  \\[4pt]
	     u
        \end{pmatrix}
        : \Pi \longrightarrow E.
\end{equation}
is locally Lipschitz continuous because of the continuous mapping $H^{1}(\gw) \hookrightarrow L^6(\gw)$ for spatial domain with $\dim (\gw) = n \leq 3$.

Below we may simply write the column vector $g(t)$ as $(u(t, \cdot), v(t, \cdot ), w(t, \cdot), \rho (t, x))$ and write $g_0 = (u_0, v_0, w_0, \rho_0)$. We shall consider the weak solution \cite{CV, SY} of this initial value problem \eqref{pb} defined below, as the basic setting.

\begin{definition} \label{Dwksn}
	A function $g(t, x), (t, x) \in [0, \tau] \times \gw$, is called a \emph{weak solution} to the initial value problem \eqref{pb}, if the following conditions are satisfied:
	
	\textup{(i)} $\frac{d}{dt}\, (g, \, \xi) = ( Ag, \xi ) + ( f(g), \,\xi )$ for almost every $t \in [0, \tau]$ and any $\xi \in C_0^\infty (\gw, \mathbb{R}^4)$;
	
	\textup{(ii)} $g(t, \cdot)  \in  C ([0, \tau]; E) \cap L^2 ([0, \tau], \Pi)$ such that $g(0) = g_0$.
	
\noindent	
Here the differential equation is satisfied in the distribution sense.	
\end{definition}

\begin{lemma} \label{Lwn}
	For any given initial state $g_0 \in E$, there exists a unique weak solution $g(t; g_0) = (u(t), v(t), w(t), \rho (t)), \, t \in [0, T)$, for some $T > 0$, of the initial value problem \eqref{pb}, which satisfies
\begin{equation} \label{soln}
    	g \in C([0, T); E) \cap C^1 ((0, T); E) \cap L_{loc}^2 ([0, T), \Pi).
\end{equation}
Any weak solution $g(t; g_0)$ becomes a strong solution for $t > 0$, which satisfies
\beq \bl{stsl}
	g \in C([t_0, T); \Pi) \cap C^1 ((t_0, T); \Pi)
\eeq
for any\, $t_0 \in (0, T)$. All the weak solutions have the continuously dependence property on the initial data in the state space $E$.
\end{lemma}
\begin{proof}
The existence and uniqueness of a weak solution local in time can be proved by the Galerkin approximation method for the PDE together with the basic existence theorem for ODE, based on the estimates similar to what we shall present in Section 3 and by the Lions-Magenes type of weak compactness argument \cite{CV, SY}. The statement about strong solution follows from the parabolic regularity \cite{SY} of the evolutionary equations in \eqref{pb}. 
\end{proof}

The goal of this work is to prove the existence of a unique global attractor for the dynamical system generated by this problem \eqref{pb}. The global attractor qualitatively characterizes the longtime and global dynamics in terms of asymptotically permanent patterns of all the solution trajectories of the system. We refer to \cite{CV, SY} for the theory details of infinite dimensional dynamical systems or called semiflow (if time $t \geq 0$). Here just list a few concepts for clarity.

\begin{definition} \label{Dabs}
Let $\{S(t)\}_{t \geq 0}$ be a semiflow on a Banach space $\ms{X}$. A bounded set $B^*$ of $\ms{X}$ is called an absorbing set for this semiflow, if for any given bounded subset $B \subset \ms{X}$ there is a finite time $T_B \geq 0$ such that $S(t)B \subset B^*$ for all $t > T_B$.
\end{definition}

\begin{definition} \label{Dasmp}
A semiflow $\{S(t)\}_{t \geq 0}$ on a Banach space $\ms{X}$ is called asymptotically compact, if for any bounded sequence $\{z_n \}$ in $\ms{X}$ and any monotone increasing sequences $0 < t_n \to \infty$, there exist subsequences $\{z_{n_k}\}$ of $\{z_n \}$ and $\{t_{n_k}\}$ of $\{t_n\}$ such that $\lim_{k \to \infty} S(t_{n_k}) z_{n_k}$ exists in $\ms{X}$. 
\end{definition}

\begin{definition}[Global Attractor] \label{Dgla}
A set $\mathscr{A}$ in a Banach space $\ms{X}$ is called a global attractor for a semiflow $\csg$ on $\ms{X}$, if the following two properties are satisfied:

(i) $\mathscr{A}$ is a nonempty, compact, and invariant set in the space $\ms{X}$,
$$
	S(t) \ms{A} = \ms{A}, \quad t \geq 0.
$$

(ii) $\mathscr{A}$ attracts any given bounded set $B \subset \ms{X}$ in the sense 
$$
	\text{dist}_{\ms{X}} (S(t)B, \mathscr{A}) = \sup_{x\, \in \, B} \inf_{y \,\in \,\mathscr{A}} \| S(t)x - y \|_{\ms{X}} \to 0, \;\;  \text{as} \; \; t \to \infty.
$$
\end{definition}

\begin{proposition}\cite{CV, SY}  \label{L:basic}
Let $\{S(t)\}_{t\geq 0}$ be a semiflow on a Banach space $\ms{X}$. If the following two conditions are satisfied\textup{:}

\textup{(i)} there exists a bounded absorbing set $B^* \subset \ms{X}$ for $\{S(t)\}_{t\geq 0}$, and

\textup{(ii)} the semiflow $\{S(t)\}_{t\geq 0}$ is asymptotically compact on $\ms{X}$,

\noindent
then there exists a unique global attractor $\ms{A}$ in $\ms{X}$ for the semiflow $\{S(t)\}_{t\geq 0}$ and 
\beq \bl{glatr}
        \ms{A} = \bigcap_{\tau \,\geq \,0} \; \overline{\bigcup_{t \,\geq \,\tau} \, (S(t)B^*)}.
\eeq
\end{proposition}

The Young's inequality in a general form will be used: For any nonnegative numbers $x$ and $y$, if $\frac{1}{p} + \frac{1}{q} = 1$, one has
\beq \bl{Yg}
	x\,y  \leq \ve x^p + C(\ve, p)\, y^q, \qquad C(\ve, p) = \ve^{-q/p},
\eeq
where constant $\ve > 0$ can be arbitrarily small. 

\section{\textbf{Uniform Estimates and Absorbing Dynamics}}

The new feature in this four-dimensional memristive Hindmarsh-Rose neuron model \eqref{ueq} - \eqref{peq} is the product coupling of the nonlinear memductance term $k_1 \vp (\rho) u$ in the membrane potential equation \eqref{ueq}. In this section we first prove the global existence of all the weak solutions in time of the initial value problem \eqref{pb}. Through careful and sophisticated maneuver of uniform inequality estimates, it will be shown that there exists an absorbing set in the state space $E$ for the solution semiflow. This dissipative dynamics result is valid without any conditions on all the 14 biological parameters in the model equations as naturally described. 

\begin{theorem} \label{T1}
For any given initial state $g_0 = (u_0, v_0, w_0, \rho_0) \in E$, there exists a unique global weak solution in time, $g(t) = (u(t), v(t), w(t), \rho (t)), \, t \in [0, \infty)$, of the initial value problem \eqref{pb} for the diffusive Hindmarsh-Rose equations with memristor \eqref{ueq}-\eqref{peq}. The weak solution turns out to be a strong solution on the interval $(0, \infty)$. 
\end{theorem}

\begin{proof}
Taking the $L^2$ inner-product $\inpt{\eqref{ueq}, C_1 u(t)}$ with a constant $C_1 > 0$, we get
\begin{equation} \label{u}
	\begin{split}
	\frac{C_1}{2} \frac{d}{dt} \|u \|^2 & + C_1 \eta \| \nabla u\|^2 = \int_\gw C_1 (au^3 -bu^4 + uv - uw +J_e u - k_1 \vp (\rho)u^2)\, dx \\
	& =  \int_\gw C_1 (au^3 -bu^4 + uv - uw + J_e u - k_1 \left( c + \ga \rho + \delta \rho^2) u^2 \right)\, dx.
	\end{split}
\end{equation}
Taking the $L^2$ inner-products $\inpt{\eqref{veq}, v(t)}$ and  $\inpt{\eqref{weq}, w(t)}$ and by Young's inequality \eqref{Yg}, we have
\begin{equation} \label{v}
	\begin{split}
	&\frac{1}{2} \frac{d}{dt} \|v \|^2 = \int_\gw (\ap v - \gb u^2 v - v^2)\, dx \\
	\leq &\int_\gw \left(\ap v +\frac{1}{2} (\gb^2 u^4 + v^2) - v^2\right) dx = \int_\gw \left(\ap v +\frac{1}{2} \gb^2 u^4 - \frac{1}{2} v^2\right) dx \\
	\leq & \int_\gw \left(2\ap^2 + \frac{1}{8} v^2 +\frac{1}{2} \gb^2 u^4 - \frac{1}{2} v^2\right) dx = \int_\gw \left(2\ap^2 +\frac{1}{2} \gb^2 u^4 - \frac{3}{8} v^2\right) dx
	\end{split}
\end{equation}
and
\begin{equation} \label{w}
	\begin{split}
	&\frac{1}{2} \frac{d}{dt} \|w \|^2 = \int_\gw (q (u - u_e)w - rw^2)\, dx  \\
	\leq & \int_\gw \left(\frac{q^2}{2r} (u - u_e)^2 + \frac{1}{2} r w^2 - r w^2 \right) dx \leq \int_\gw \left(\frac{q^2}{r} (u^2 + u_e^2) - \frac{1}{2} r w^2 \right) dx.
	\end{split}
\end{equation}
Taking the $L^2$ inner-products $\inpt{\eqref{peq}, \rho (t)}$ and we get
\beq \bl{rh}
	\frac{1}{2} \frac{d}{dt} \|\rho \|^2 = \int_\gw (u\rho - k_2 \rho^2)\, dx \leq \int_\gw \left( \frac{1}{2k_2} u^2 - \frac{k_2}{2} \rho^2 \right)\, dx \leq \int_\gw \left(u^4 - \frac{k_2}{2} \rho^2 \right) dx + \frac{1}{4k_2^2}\, |\gw |.
\eeq

Now Choose the scaling constant in \eqref{u} to be $C_1 = (\gb^2 + 5)/b$ so that 
\beq \bl{C1b}
	- \int_\gw C_1 b u^4 \, dx + \int_\gw \frac{1}{2} \gb^2 u^4\, dx \leq \int_\gw (- 5 u^4)\, dx.
\eeq 
Then we estimate the following terms on the right-hand side of \eqref{u} by using Young's inequality \eqref{Yg} in an appropriate way:
\begin{gather*}
	\int_\gw C_1 au^3\, dx \leq \frac{3}{4} \int_\gw u^4\, dx + \frac{1}{4}\int_\gw (C_1 a)^4 \, dx \leq \int_\gw u^4\, dx + (C_1 a)^4 |\gw|, \\
	\int_\gw C_1 (uv - uw + J_e u)\, dx \leq \int_\gw \left(2(C_1 u)^2 + \frac{1}{8} v^2 + \frac{(C_1 u)^2}{r} + \frac{1}{4} r w^2 + C_1 u^2 + C_1J_e^2 \right) dx \\
	\leq \int_\gw u^4 \, dx + \left[C_1^2 \left(2 +\frac{1}{r}\right) + C_1\right]^2 |\gw | + \int_\gw \left( \frac{1}{8} v^2 + \frac{1}{4} r w^2 + C_1J_e^2 \right) dx,
\end{gather*}
and by completing square,
\beq \bl{mterm}
	\int_\gw (- C_1 k_1 (c + \ga \rho + \delta \rho^2) u^2)\, dx \leq C_1 k_1 \left(| \,c\, | + \frac{\ga^2}{4 \de}\right) \int_\gw u^2 \, dx.
\eeq
In \eqref{w}, we have
\begin{gather*}
	\int_\gw \frac{q^2}{r} u^2 \, dx \leq \int_\gw \left(\frac{u^4}{2} + \frac{q^4}{2r^2}\right) dx  \leq \int_\gw u^4\, dx + \frac{q^4}{r^2} |\gw|.
\end{gather*}
Substitute the above term estimates with \eqref{C1b} and \eqref{mterm} into \eqref{u} - \eqref{rh}. We obtain
\beq \label{g}
	\begin{split}
	&\frac{1}{2} \frac{d}{dt} (C_1 \|u\|^2 +  \|v\|^2 +  \|w\|^2 + \| \rho \|^2) + C_1 \eta \|\nb u\|^2  \\
	\leq & \int_\gw C_1 (au^3 -bu^4 + uv - uw + J_e u - k_1 (c + \gamma \rho + \delta \rho^2) u^2)\, dx \\
	& + \int_\gw \left[2\ap^2 +\frac{1}{2} \gb^2 u^4 - \frac{3}{8} v^2\right] + \left[\frac{q^2}{r} (u^2 + u_e^2) - \frac{r}{2} w^2 \right] dx + \int_\gw \left[u^4 - \frac{k_2}{2}\rho^2 \right] dx + \frac{| \gw |}{4k_2^2}  \\
	\leq & \int_\gw (4 - 5)u^4\, dx + C_1 k_1 \left(| \,c\, | + \frac{\ga^2}{4 \de}\right) \int_\gw u^2 \, dx  \\
	& + \int_\gw \left(\frac{1}{8} - \frac{3}{8}\right) v^2\, dx + \int_\gw \left(\frac{1}{4} - \frac{1}{2} \right) rw^2\, dx \\
	& + |\gw | \left[ (C_1 a)^4 + C_1 J_e^2  + \left(C_1^2 \left(2 +\frac{1}{r}\right) + C_1\right)^2 + 2\ap^2 + \frac{q^2 u_e^2}{r} + \frac{q^4}{r^2} + \frac{1}{4k_2^2}\right] \\
	= &\, - \int_\gw \left[ u^4 - C_1 k_1 \left[|c| + \frac{\ga^2}{4 \de} \right]u^2 + \frac{1}{4} v^2 + \frac{r}{4} w^2 + \frac{k_2}{2} \rho^2 \right]dx + C_2 |\gw |\, ,
	\end{split}
\eeq
where $C_2 > 0$ is the constant given by 
\beq \bl{C2}
	C_2 = (C_1 a)^4 + C_1 J_e^2  + \left[C_1^2 \left(2 +\frac{1}{r}\right) + C_1\right]^2 + 2\ap^2 + \frac{q^2 u_e^2}{r} + \frac{q^4}{r^2} + \frac{1}{4k_2^2} \, .
\eeq
Note that
$$
	u^4 - C_1 k_1 \left[|c| + \frac{\ga^2}{4 \de} \right] u^2 \geq \frac{1}{2} u^4 - 2 C_1^2 k_1^2 \left[|c| + \frac{\ga^2}{4 \de} \right]^2.
$$
Thus \eqref{g} yields the following uniform grouping estimate for all the solutions of the memristive Hindmarsh-Rose system \eqref{pb},
\beq \label{Es}
	\begin{split}
	&\frac{d}{dt} \left(C_1 \|u(t)\|^2 + \|v(t)\|^2 + \|w(t)\|^2 + \| \rho (t)\|^2 \right)  + C_1 \eta \|\nb u\|^2  \\
	+ & \int_\gw \left(u^4(t, x) + \frac{1}{2} v^2 (t, x) + \frac{r}{2} w^2 (t, x) + k_2 \rho^2 (t, x) \right) dx  \leq C_3 |\gw|, 
	\end{split}
\eeq
where 
\beq \bl{C3}
	C_3 = 2C_2 + 4C_1^2 k_1^2 \left[|c| + \frac{\ga^2}{4 \de} \right]^2,
\eeq
for $t \in I_{max} = [0, T_{max})$, which is the maximal time interval of solution existence. Furthermore, since 
$$ 
	u^4 \geq \frac{C_1}{2} u^2 - \frac{C_1^2}{4}\, ,
$$
it follows from \eqref{Es} that
\begin{equation*}
	\begin{split}
	&\frac{d}{dt} \left(C_1 \|u(t)\|^2 + \|v(t)\|^2 + \|w(t)\|^2 + \| \rho (t)\|^2 \right)  + C_1 \eta \,\|\nb u\|^2  \\
	+ &\, \int_\gw \frac{1}{2} \left(C_1 u^2(t, x) + v^2 (t, x) + r w^2 (t, x) + k_2^2 \rho^2 (t, x)\right) dx \leq \left(C_3 + \frac{C_1^2}{4}\right) |\gw |.
	\end{split}
\end{equation*}
Set $\gl = \frac{1}{2} \min \{1, r, k_2 \}$. Then we have, for $t \in [0, T_{max})$, 
\begin{equation} \label{Gq}
	\begin{split}
	& \frac{d}{dt} (C_1 \|u(t)\|^2 + \|v(t)\|^2 + \|w(t)\|^2 + \|\rho (t)\|^2)  + C_1 \eta \, \| \nabla u(t) \|^2   \\
	+ \gl \, ( & C_1 \| u(t)\|^2 + \| v(t) \|^2 + \|w(t)\|^2 + \| \rho (t)\|^2) \leq \left(C_3 + \frac{C_1^2}{4}\right) |\gw | \, .
	\end{split}
\end{equation}

Apply the Gronwall inequality to the differential inequality \eqref{Gq}. We obtain
\beq \label{ds}
	\begin{split}
	& C_1 \|u(t)\|^2 + \|v(t)\|^2 + \|w(t)\|^2 + \| \rho (t) \|^2  \\[5pt]
	\leq &\, e^{- \gl t} (C_1 \|u_0\|^2 + \|v_0\|^2 + \|w_0\|^2 + \| \rho_0\|^2) + M |\gw |,  \quad \text{for} \;\,  t \in [0, \infty),
	\end{split}
\eeq 
where 
\beq \bl{M}
	M = \frac{1}{\gl}\left(C_3 + \frac{C_1^2}{4}\right) = \frac{2}{\min \{1, r, k_2\}} \left(C_3 + \frac{C_1^2}{4}\right).
\eeq
The estimate \eqref{ds} shows that all the weak solutions will never blow up at any finite time because it is uniformly bounded. Namely, for all $t \in [0, \infty)$, it holds that 
$$
	C_1 \|u(t)\|^2 + \|v(t)\|^2 + \|w(t)\|^2 +\|\rho (t)\|^2 \leq C_1 \|u_0\|^2 + \|v_0\|^2 + \|w_0|^2 + \|\rho_0\|^2 + M |\gw |.
$$
Therefore the weak solution of the initial value problem \eqref{pb} formulated from the diffusive Hindmarsh-Rose equations with memristor \eqref{ueq} - \eqref{peq} exists globally in time for any initial data in the state space $E$ and the time interval of maximal existence will always be always $[0, \infty)$. The proof is completed.
\end{proof}

The global existence and uniqueness of the weak solutions as well as their continuous dependence on the initial data enable us to define the solution semiflow \cite{SY} of the diffusive Hindmarsh-Rose equations with memristor \eqref{ueq}-\eqref{peq} on the state space $E$ as follows:
$$
	S(t): g_0 \longmapsto g(t; g_0) = (u(t, \cdot), v(t, \cdot), w(t, \cdot), \rho (t, \cdot)), \quad  g_0 \in E, \; t \geq 0,
$$
where $g(t; g_0)$ is the weak solution with $g(0) = g_0$. We call this semiflow $\{S(t)\}_{t \geq 0}$ the \emph{memristive Hindmarsh-Rose semiflow} associated with the system \eqref{pb}. 

The next result exhibits the globally dissipative dynamics of this solution semiflow in the state space $E$.

\begin{theorem} \label{T2}
	There exists a bounded absorbing set in the space $E$ for the memristive Hindmarsh-Rose semiflow $\{S(t)\}_{t \geq 0}$, which is a bounded ball 	
\beq \label{abs}
	B_E = \{ g \in E: \| g \|^2 \leq K\}
\eeq 
where 
\beq \bl{K}
	K = \frac{M |\gw |}{\min \{C_1, 1\}} + 1.
\eeq
\end{theorem}

\begin{proof}
From the globally uniform estimate \eqref{ds} in the proof of Theorem \ref{T1} we see that 
\beq \label{lsp2}
	\limsup_{t \to \infty} \; (\|u(t)\|^2 + \|v(t)\|^2 + \|w(t)\|^2 + \|\rho (t)\|^2) < K = \frac{M |\gw |}{\min \{C_1, 1\}} + 1
\eeq
for all weak solutions of \eqref{pb} with any initial state $g_0 \in E$. 

Moreover, for any given bounded set $B = \{g \in E: \|g \|^2 \leq \widehat{R}\}$ in $E$, where $\widehat{R}$ is a finite positive number, there exists a finite time 
$$
	T_0 (B) = \frac{1}{\gl} \log^+ (\widehat{R} \, \max \{C_1, 1\})
$$
such that  all the solutions $g(t; g_0)$ of \eqref{pb} satisfy $\|u(t)\|^2 + \|v(t)\|^2 + \|w(t)\|^2 + \|\rho(t)\|^2 < K$ for time $t  > T_0 (B)$ and any initial state $g_0 \in B$. Thus, by Definition \ref{Dabs}, the bounded ball $B_E$ is an absorbing set for the memristive Hindmarsh-Rose semiflow $\{S(t)\}_{t \geq 0}$ in the phase space $E$ and it is a dissipative dynamical system.
\end{proof}

\section{\textbf{Higher-Order Dissipativity of Memristive Hindmarsh-Rose Semiflow}}

In this section, we explore higher-order dissipativity of the memristive Hindmarsh-Rose semiflow for the $u$-component in space $L^4 (\gw)$. It will pave the way to prove the asymptotic compactness of this semiflow in the next section, which is the key condition for the existence of a global attractor in an infinite-dimensional state space.

\begin{theorem} \bl{Tu}
There exists a constant $Q > 0$ independent of any initial state, such that the $u$-component of the memristive HIndmarsh-Rose semiflow $\{S(t)\}_{t \geq 0}$ has the uniform dissipative property that for any given bounded set $B \subset E$ there is a finite time $T^u_B > 1$ and 
\beq \bl{Ku}
	\sup_{g_0 \in B}\, \int_\gw u^4 (t, x) \, dx \leq Q, \quad \text{for} \;\;  t > T^u_B.
\eeq
\end{theorem}

\begin{proof}
Take the $L^2$ inner-product $\langle \eqref{ueq}, u^3(t, \cdot) \rangle$ and use Young's inequality \eqref{Yg} appropriately to split the product terms in the integral below.  For $t > 0$ we have
\begin{equation}  \bl{uvq}
	\begin{split} 
	\frac{1}{4}&\, \frac{d}{dt} \|u(t)\|^{4}_{L^{4}} + 3 \eta \|u \nb u \|^2_{L^2} = \int_\gw [au^{5} - bu^{6} + u^{3} (v - w - J_e) - k_1 (c + \ga \rho + \de \rho^2) u^4] dx \\[2pt]
	\leq &\,\int_\gw \left[ \left(C_{a, b} + \frac{1}{4} bu^6\right) - b u^{6} + \left(\frac{1}{4} bu^6 + C_{b} (v^{2} + w^{2} + J_e^2) \right) + k_1 \left(| c| + \frac{\ga^2}{\de}\right) u^4 \right] dx \\
	\leq &\,\int_\gw \left[ \left(C_{a, b} + \frac{1}{4} bu^6\right) - b u^{6} + \left(\frac{1}{4} bu^6 + C_{b} (v^{2} + w^{2} + J_e^2) \right) \right] dx \\
	 &\, + \int_\gw \left[\frac{1}{4} bu^6 + \frac{16 \,k_1^3}{b^2}  \left(| c| + \frac{\ga^2}{\de}\right)^3 \right] dx. \\
	 \leq &\, - \frac{1}{4} \int_\gw b u^6\, dx + C_{b} \int_\gw (v^2 + w^2) \, dx + | \gw | \left[C_{a,b} + C_{b} J_e^2 + \frac{16\, k_1^3}{b^2}  \left(| c| + \frac{\ga^2}{\de}\right)^3 \right],
	\end{split}
\end{equation} 
where $C_{a,b}$ and $C_b$ are positive constants depending on $a, b$ and on $b$, respectively. By Theorem \ref{T2} and \eqref{lsp2}, for any given bounded set $B \subset E$, there is a finite time $\tau_B > 0$ such that
$$
	C_{b} \int_\gw (v^2(t, x) + w^2 (t, x)) \, dx \leq C_{b} K, \quad  \text{for} \;\,  t > \tau_B.
$$
Since $u^6 + 1 \geq u^4$, from \eqref{uvq} and the above inequality it follows that 
\beq \bl{u4}
	\frac{d}{dt} \|u(t)\|^{4}_{L^{4}} + \int_\gw bu^4\, dx \leq  | \gw | \left[b + C_{a,b} + C_{b} (K + J_e^2) + \frac{16\, k_1^3}{b^2}  \left(| c| + \frac{\ga^2}{\de}\right)^3 \right].
\eeq
Apply the Gronwall inequality to \eqref{u4} and it yields 

\beq \bl{ub}
	\|u(t)\|_{L^4}^4 \leq e^{-bt} \|u(t_0)\|_{L^4}^4 + \frac{1}{b} | \gw | \left[b + C_{a,b} + C_{b, r} (K + J_e^2) + \frac{16\, k_1^3}{b^2}  \left(| c| + \frac{\ga^2}{\de}\right)^3 \right], 
\eeq
for $t \geq t_0 > \tau_B$.

It remains to bound the $L^4$ norm of the initial state $u(t_0)$.  By Lemma \ref{Lwn}, for any weak solution of the memristive Hindmarsh-Rose evolutionary equation \eqref{pb}, the $u$-component has the regularity
$$
	u(t, \cdot) \in H^1 (\gw) \subset L^4 (\gw), \quad \text{for} \;\; t  > 0.
$$
One can integrate \eqref{Gq} over the time interval $(0, t]$ to get
$$
	 \int_0^t C_1 (\eta \|\nb u(s)\|^2 + \gl \|u(s)\|^2)\, ds \leq \max \{C_1, 1\} \|g_0\|^2 + t \gl M | \gw|, \; \; t \geq 0,
$$
where the constant $M$ is given in \eqref{M}. It follows that, for $t = 1$,
\begin{equation} \bl{H1b}
	\int_0^1 C_1 \|u(s)\|^2_{H^1}\, ds \leq \frac{1}{\min \{\eta, \gl \}} \left(\max \{C_1, 1\} \|g_0\|^2 +  \gl M | \gw|\right).
\end{equation}
Hence for any given bounded set $B \subset E$ and any initial state $g_0 \in B$, there exists a time point $t_0 \in (0, 1)$ such that 
\beq \bl{ut0}
	\|u(t_0)\|^2_{L^4} \leq \widehat{C} \|u(t_0)\|_{H^1}^2 \leq \frac{\widehat{C}}{C_1 \min \{\eta, \gl \}} \left(\max \{C_1, 1\} \interleave B \interleave^2 + \gl M | \gw|\right) 
\eeq 
where $\widehat{C}$ is the embedding coefficient of $H^1 (\gw)$ into $L^4 (\gw)$ and $\interleave B \interleave = \sup_{g_0 \in B} \|g_0\|$.

Finally, combining the inequalities \eqref{ub} and \eqref{ut0}, we conclude that for any given bounded set $B \subset E$ and any initial state $g_0 \in B$, there exists a finite time 
$T^u_B > \max \{1, \tau_B\}$ such that the target result of the inequality \eqref{Ku} is valid with the uniform ultimate bound
\beq \bl{Q}
	\begin{split}
	Q &= \left[\frac{\widehat{C}}{C_1 \min \{\eta, \gl \}} \left(\max \{C_1, 1\} \interleave B \interleave^2 + \gl M | \gw|\right) \right]^2   \\
	&\, + \frac{1}{b} | \gw | \left[b + C_{a,b} + C_{b, r} (K + J_e^2) + \frac{16\, k_1^3}{b^2}  \left(| c| + \frac{\ga^2}{\de}\right)^3 \right].
	\end{split}
\eeq
The proof is completed.
\end{proof}

\begin{corollary} \bl{Cu}
The component solution $u(t, x)$ of the memristive Hindmarsh-Rose semiflow $\{S(t)\}_{t \geq 0}$ has the dissipative property in the space $L^4 (\gw)$. For any given bounded set $B \subset E$, there exists a finite time $T^u_B > \max \{1, \tau_B\}$ such that 
\beq \bl{ucp}
	\bigcup_{t \, > \,T^u_B} \left(\bigcup_{g_0\, \in \,B} u(t, \cdot)\right) \, \text{is bounded in} \; L^4 (\gw) \; \text{and precompact in} \; L^2 (\gw).
\eeq
\end{corollary}
\begin{proof}
 	Since the Sobolev embedding $L^4 (\gw) \hookrightarrow L^2 (\gw)$ is compact for the bounded region $\gw$, it is a direct consequence that the set in \eqref{ucp} is precompact in $L^2(\gw)$.
\end{proof}

\section{\textbf{Asymptotic Compactness of Memristive Hindmarsh-Rose Semiflow}}

In this section, we prove the asymptotic compactness, cf. Definition \ref{Dasmp}, of the memristive Hindmarsh-Rose solution semiflow $\{S(t)\}_{t \geq 0}$. This is a challenging issue as the components $v(t, x), w(t, x), \rho(t, x)$ of the memristive Hindmarsh-Rose equations formulated in \eqref{pb} do not have the regularized property in $x$ as time evolves. 

The leverage we use to tackle the asymptotic compactness is the Kolmogorov-Riesz compactness Theorem below shown in \cite[Theorem 5]{HO}. 

\begin{lemma} \bl{KR}
	Let $1 \leq p < \infty$ and $\gw \subset \mathbb{R}^n$ be a bounded domain with locally Lipschitz continuous boundary. A subset $\mathcal{F}$ in the function space $L^p (\gw)$ is precompact if and only if the following two conditions are satisfied\textup{:}

	\textup{1)} $\mathcal{F}$ is a bounded set in $L^p (\gw)$.
	
	\textup{2)} For every $\ve > 0$, there is some positive number $d > 0$ such that, for all $f \in \mathcal{F}$ and $y \in \mathbb{R}^n$ with $| y | < d $, it holds that
$$
	\int_\gw | f(x + y) - f(x)|^p \, dx < \ve^p.
$$
It is a convention that $f(x) = 0$ for $x \in \mathbb{R}^n \backslash \gw$.
\end{lemma}

\begin{theorem} \bl{AC}
	The solution semiflow $\{S(t)\}_{t \geq }$ generated by the diffusive Hindmarsh-Rose equations with memristor \eqref{pb} is asymptotically compact in the state space $E$.
\end{theorem}

\begin{proof}
As a setup in this proof, for any given bounded set $B \subset H$, let $T^0_B > 0$ be a finite time such that for any initial state $g_0 \in B$ one has 
\beq \bl{T0B}
	\|g(t; g_0)\|^2 = \|u(t)\|^2 + \|v(t)\|^2 + \|w(t)\|^2 + \|\rho(t)\|^2 \leq K, \quad \text{for} \;\, t > T^0_B ,
\eeq
where the constant $K$ is given in \eqref{K} of Theorem \ref{T2}.

Step 1. First of all, \eqref{ucp} in Corollary \ref{Cu} has shown that the $u$-component of this memristive Hindmarsh-Rose semiflow is ultimately uniform bounded in the space $L^4 (\gw)$, which is compactly imbedded in $L^2 (\gw)$. Hence, according to Definition \ref{Dasmp}, the $u$-component of this semiflow is asymptotically compact in the space $L^2 (\gw)$.  

We now deal with the component functions $v(t, x)$, which is coupled with $u(t, x)$ in the nonlinear differential equation \eqref{veq}. By the variation-of-constant formula, we have the expressions: for $t \geq t_0 \geq 0$,
\begin{equation} \bl{vcw}
	v(t, x) = e^{-t} v(t_0) + \int_{t_0}^t e^{-(t-s)} (\alpha - \beta u^2)\, ds  \leq \alpha + e^{-t} v(t_0) - \beta \int_{t_0}^t e^{-(t-s)} u^2(s, x)\, ds.
\end{equation} 
In view of \eqref{ucp} in Corollary \ref{Cu} and that the embedding $L^4 (\gw) \hookrightarrow L^3 (\gw)$ is compact, according to Lemma \ref{KR}, we can assert that for any $\ve > 0$, there is some $d > 0$ such that, for any given bounded set $B \subset E$ and all $g_0 \in B$, and for $y \in \mathbb{R}^3$ with $| y | < d $, there exists a finite time $T_B \geq T^0_B$ such that 
\beq \bl{utx}
	\int_\gw | u(t, x + y) - u(t, x)|^3 \, dx < \ve^3, \quad \text{for all} \;\, t >T_B.
\eeq
Using the H\"{o}lder inequality we can infer that, for any $t > T_B$ and any $g_0 \in B$,
\begin{equation} \bl{ky}
	\begin{split}
	& \int_\gw |v(t, x+y) - v(t, x)|^2 dx = 2e^{- 2(t - T_B)} \int_\gw |v(T_B, x + y) - v(T_B, x)|^2 \,dx \\
	+ &\, 2\beta^2 \int_\gw \left(\int_{T_B}^t e^{- (t-s)} |u^2(s, x+ y) - u^2 (s, x)| \, ds \right)^2 dx  \\
	\leq &\, 4\, e^{- 2(t-T_B)} \|v (T_B)\|^2    \\
	+ &\, 2 \gb^2 \int_\gw \left(\int^t_{T_B} e^{- (t-s)}\, ds\right) \left(\int^t_{T_B} e^{- (t - s)} |u^2(s, x+ y) - u^2 (s, x)|^2\, ds\right) dx  \\
	\leq &\, 4\, e^{- 2(t-T_B)} \|v (T_B)\|^2    \\
        + &\, 2\beta^2 \int_{T_B}^t e^{- (t-s)} \int_\gw |u(s, x+ y) - u (s, x)|^2 |u(s, x+ y) + u (s, x)|^2 \,dx\, ds    \\
        \leq &\, 4\, e^{-( 2(t-T_B)} \|v (T_B)\|^2  \\
        + &\, 2\beta^2 \int_{T_B}^t e^{- (t-s)} \|(u(s, x+y) - u(s, x))^2\|_{L^{3/2}} \|(u(s, x+y) + u(s,x))^2\|_{L^3}\, ds   \\
        \leq &\, 4\, e^{- 2(t-T_B)} \|v (T_B)\|^2  \\
        + &\, 4 \beta^2 \int_{T_B}^t e^{- (t-s)} \|u(s, x+y) - u(s, x)\|^{2}_{L^3} \left(\|u(s, x+y)\|_{L^6}^{2} + \|u(s, x)\|_{L^6}^{2}\right) ds \\
       \leq &\, 4 e^{- 2(t-T_B)} \|v (T_B)\|^2 + 8 \beta^2 \int_{T_B}^t e^{- (t-s)} \|u(s, x+y) - u(t, x)\|^{2}_{L^3} \|u(s, x)\|_{L^6}^{2}\, ds \\
        \leq &\, 4 e^{- 2(t-T_B)} \|v (T_B)\|^2 + 8 \beta^2 \int_{T_B}^t e^{- (t-s)} \|u(s, x+y) - u(t, x)\|^{2}_{L^3} (\|u(s, x)\|_{L^6}^{6} + 1)\, ds,
        	\end{split}
\end{equation} 
wherein we have used following intermediate steps: $\int^t_{T_B} e^{- (t-s)}\, ds < 1 \, \text{for} \, t > T_B,$ and
\begin{gather*}
	\|(u(s, \cdot + y) - u(s, \cdot))^2\|_{L^{3/2}} = \| u(s, \cdot + y) - u(s, \cdot)\|^2_{L^3},  \\[2pt]
	\|(u(s, \cdot +y) + u(s, \cdot))^2\|_{L^3} \leq \|2u^2 (s, \cdot + y) + 2u^2 (s, \cdot)\|_{L^3} \leq 4 \|u(s, \cdot)\|^2_{L^6}, \\
	\|u(s, \cdot)\|^2_{L^6} \leq \frac{1}{3} \|u(s, \cdot)\|^6_{L^6} + \frac{2}{3} \leq \|u(s, \cdot)\|^6_{L^6} + 1.
\end{gather*}

On the other hand, taking account of \eqref{lsp2}, \eqref{u4} and \eqref{T0B}, we can do the following integration
\begin{equation}  \bl{eu}
	\begin{split} 
	&\int_{T_B}^{t}  e^{-(t - s)} \frac{d}{dt} \|u(s)\|^{4}_{L^{4}} ds + \int^{t}_{T_B} e^{-(t - s)} \int_\gw b u^6(s, x)\, dx \, ds  \\
	\leq &\, \int_{T_B}^{t} 4e^{-(t - s)} | \gw | \left[C_{a,b} + C_{b} (K + J_e^2) + \frac{16\, k_1^3}{b^2}  \left(| c| + \frac{\ga^2}{\de}\right)^3 \right] ds \\
	\leq &\, 4 |\gw | \left[C_{a,b} + C_{b} (K + J_e^2) + \frac{16\, k_1^3}{b^2}  \left(| c| + \frac{\ga^2}{\de}\right)^3 \right], \quad \text{for} \;\, t > T_B \,(\geq T^0_B),
	\end{split}
\end{equation} 
where
$$
	\int_{T_B}^{t}  e^{-(t - s)} \frac{d}{dt} \|u(s)\|^{4}_{L^{4}}\, ds = \|u(t)\|^{4}_{L^{4}} - e^{- (t - T_B)} \|u(T_B)\|^4_{L^4} - \int_{T_B}^{t}  e^{-(t - s)} \|u(s)\|^{4}_{L^{4}}\,ds.
$$
Thus \eqref{eu} with \eqref{Ku} shows that
\begin{equation}  \bl{EU}
	\begin{split} 
	\int^{t}_{T_B} &\, e^{-(t - s)} \int_\gw u^6(s, x)\, dx\, ds \leq \frac{1}{b}\left[ e^{- (t - T_B)} \|u(T_B)\|^4_{L^4} + \int_{T_B}^{t}  e^{-(t - s)} \|u(s)\|^{4}_{L^{4}}\,ds \right] \\
	&+ \frac{4}{b} |\gw | \left[C_{a,b} + C_{b} (K + J_e^2) + \frac{16\, k_1^3}{b^2}  \left(| c| + \frac{\ga^2}{\de}\right)^3 \right]   \\
	&\leq L = \frac{2}{b}\, Q + \frac{4}{b} |\gw | \left[C_{a,b} + C_{b} (K + J_e^2) + \frac{16\, k_1^3}{b^2}  \left(| c| + \frac{\ga^2}{\de}\right)^3 \right], \quad t > T_B.
	\end{split}
\end{equation}
Here $L > 0$ and $K$ in \eqref{K}, $Q$ in \eqref{Q} are uniform constants. 

Now combine \eqref{utx}, \eqref{ky} and \eqref{EU}. Then we can confirm that for any $\ve > 0$ there is a number $d > 0$ such that, for any given bounded set $B \subset E$ and all $g_0 \in B$, and for $y \in \mathbb{R}^3$ with $| y | < d$, it holds that 
\beq \bl{vtx}
	\begin{split}
	&\int_\gw |v(t, x+y) - v(t, x)|^2 dx \\
	\leq &\, 4 e^{- 2(t-T_B)} \|v (T_B)\|^2 + 8\, \beta^2 \int_{T_B}^t e^{-(t-s)} \,\ve^2 \, \left(\int_\gw u(s, s)^6\,dx + 1\right) ds \\
	= &\, 4 e^{- 2(t-T_B )}K + 8\, \beta^2 \ve^2 \left(L + 1\right), \quad \text{for} \; t > T_B, \; g_0 \in B,
	\end{split}
\eeq
Therefore, there exists a time 
$$
	T^v_B = T_B + \frac{1}{2} \ln \left(\frac{\ve^2}{4 K} \right)
$$ 
such that 
\beq \bl{vEp}
	\int_\gw |v(t, x+y) - v(t, x)|^2 dx < \left[ 1 + 8 \beta^2 (L + 1) \right] \ve^2, \quad t > T^v_B, \; \, g_0 \in B.
\eeq
Since $\ve > 0$ is arbitrary, according to Lemma \ref{KR}, here \eqref{vEp} implies that 
\beq \bl{vcp}
	\bigcup_{t\, > T^u_B} \left(\bigcup_{g_0\, \in \,B} v(t, \cdot)\right) \; \text{is precompact in}\; L^2 (\gw).
\eeq

  Step 2.  Next we deal with the solution component $w(t, x)$, which is linearly coupled with the component $u(t, x)$ in \eqref{weq}:
\beq \bl{wq}
	\begin{split}
	w(t, x) &= e^{- rt} w(t_0) + \int_{t_0}^t e^{-r (t-s)} q (u - u_e)\, ds   \\
	&= - \frac{q \,u_e}{r} + e^{- rt} w(t_0) + q \int_{t_0}^t e^{-r (t-s)} u(s, x) ds.
	\end{split}
\eeq
Similarly, by Lemma \ref{KR} and \eqref{ucp}, for any $\ve > 0$, there is some $d > 0$ such that, for any given bounded set $B \subset E$ and all $g_0 \in B$, there is a finite time $T_B \,( \geq T^0_B)$ and for $y \in \mathbb{R}^3$ with $| y | < d$, it holds that as in \eqref{utx}
$$
	\int_\gw | u(t, x + y) - u(t, x)|^3 \, dx < \ve^3, \quad \text{for all} \;\, t > T_B .
$$
Then we can show that, for any $g_0 \in B$,
\begin{equation} \bl{wtx}
	\begin{split}
	&\int_\gw |w(t, x+y) - w(t, x)|^2 dx \leq 4 e^{- 2r (t-T_B)} \|w (T_B )\|^2   \\
	+ &\, \frac{2 q^2}{r} \int_{T_B}^t e^{- r(t-s)} \int_\gw |u(s, x+ y) - u(s, x)|^2 dx\, ds < \left(1 + \frac{2 q^2}{r^2}\right) \ve^2, \quad t > T_B^w,
	\end{split}
\end{equation}
where H\"{o}lder inequality is used for the double integral term in \eqref{wtx} as in the first inequality in \eqref{ky}, and
$$
	T_B^w = T_B + \frac{1}{2r} \ln \left(\frac{\ve^2}{4 K} \right).
$$
Since $\ve > 0$ is arbitrary, Lemma \ref{KR} and \eqref{wtx} confirm that 
\beq \bl{wcp}
	\bigcup_{t\, > T_B^w} \left(\bigcup_{g_0\, \in \,B} w(t, \cdot)\right) \; \text{is precompact in}\; L^2 (\gw).
\eeq

   Lastly we deal with the solution component $\rho (t, x)$, which is linearly coupled with the component $u(t, x)$ in \eqref{peq}:
   
\beq \bl{rq}
	\rho (t, x) = e^{- k_2 t} \rho (t_0) + \int_{t_0}^t e^{-k_2 (t-s)} u(s,x) ds.
\eeq
Again by Lemma \ref{KR} and \eqref{ucp}, for any $\ve > 0$, there is some $d > 0$ such that, for any given bounded set $B \subset E$ and all $g_0 \in B$, there is a finite time $T_B\, ( \geq T^0_B)$ and for $y \in \mathbb{R}^3$ with $| y | < d$, it holds that as in \eqref{utx}
$$
	\int_\gw | u(t, x + y) - u(t, x)|^3 \, dx < \ve^3, \quad \text{for all} \;\, t > T_B .
$$
Then, similar to \eqref{wtx}, for any $g_0 \in B$ we have
\begin{equation} \bl{ptx}
	\begin{split}
	&\int_\gw |\rho (t, x+y) - \rho (t, x)|^2 dx \leq 4 e^{-  2k_2 (t-T_B)} \|\rho (T_B )\|^2   \\
	+ \frac{2}{k_2}  \int_{T_B}^t &\, e^{- k_2 (t-s)} \int_\gw |u(s, x+ y) - u(s, x)|^2 dx\, ds < \left( 1 + \frac{2}{k_2^2}\right) \ve^2, \quad t > T_B^\rho.
	\end{split}
\end{equation}
where 
$$
	T_B^\rho = T_B + \frac{1}{2 k_2} \ln \left(\frac{\ve^2}{4 K} \right).
$$
Consequently, since $\ve > 0$ is arbitrary, \eqref{ptx} confirms that 
\beq \bl{pcp}
	\bigcup_{t >T_B^\rho} \left(\bigcup_{g_0\, \in \,B} \rho (t, \cdot)\right) \; \text{is also precompact in}\; L^2 (\gw).
\eeq

Finally, put together \eqref{ucp}, \eqref{vcp}, \eqref{wcp} and \eqref{pcp}. We conclude that there exists a finite time
$$
	T^*_B = \max \, \{T^u_B, \, T^v_B, \, T^w_B, \, T^\rho_B\}
$$
such that all the solutions $g(t; g_0) = (u(t, \cdot), v(t, \cdot), w(t, \cdot), \rho (t, \cdot))$ started from any given bounded set $B \subset E$ satisfies the asymptotically compact property:
\beq \bl{gcp}
	\bigcup_{t\, > \,T^*_B} \left(\bigcup_{g_0\, \in \,B} g(t; g_0) \right) \;\, \text{is precompact in}\; E.
\eeq
We conclude that the memristive Hindmarsh-Rose semiflow $\{S(t)\}_{t \geq 0}$ generated by \eqref{pb} is asymptotically compact in the state space $E$. The proof is completed. 
\end{proof}

\section{\textbf{Global Attractor Existence and Regularity}}	

In this section we finally achieve the main result on the existence of a global attractor in the state space $E$ for the new proposed neuron model of the diffusive Hindmarsh-Rose equations with memristors. We shall also demonstrate the regularity property of the global attractor in the regular space $\Gamma = H^2 (\gw) \times L^\infty (\gw, \mathbb{R}^3)$. 

\begin{theorem} \bl{main}
	There exits a unique global attractor $\mathscr{A}$ in the state space $E = L^2 (\gw, \mathbb{R}^4)$ for the memristive Hindmarsh-Rose semiflow $\{S(t)\}_{t \geq 0}$ generated by the diffusive Hindmarsh-Rose equations with memristors \eqref{ueq} - \eqref{peq}.
\end{theorem}

\begin{proof}
	Since Theorem \ref{T2} shows that there exists a bounded absorbing set $B_E = \{ g \in E: \| g \|^2 \leq K\}$ and Theorem \ref{AC} shows that the solution semiflow $\{S(t)_{t \geq 0}\}$ generated by the diffusive Hindmarsh-Rose equations with memristors \eqref{pb} is asymptotically compact in the space $E$, the two conditions required in Proposition \ref{L:basic} are satisfied. Therefore, by Proposition \ref{L:basic}, there exists a unique global attractor $\mathscr{A}$ for this semiflow in the space $E$ and this attractor is given by
\beq \bl{GA}
	         \ms{A} = \bigcap_{\tau\, \geq \,0} \; \overline{\bigcup_{t \, \geq \,\tau} \, (S(t) B_E)}.
\eeq
where the closure is take in the space $E$. 
\end{proof}

The following two results provide the regularity information about the global attractor $\ms{A}$ for this memristive and diffusive Hindmarsh-Rose neuron model. 

\begin{lemma} \label{Lcb}
	For space dimension $\dim\, (\gw) = 1$, the $u$-projection of the global attractor $\mathscr{A}$ in \eqref{GA} of the memristive Hindmarsh-Rose semiflow $\{S(t)\}_{t \geq 0}$ is a bounded set in the continuous function space $C (\overline{\gw})$ as well as in $L^\infty (\gw)$.
\end{lemma}

\begin{proof}
By Definition \ref{Dgla}, the global attractor $\mathscr{A}$ is an invariant set so that 
\beq \bl{Sta}
	S(t) \mathscr{A} = \mathscr{A} \subset B_E, \quad \text{for all} \;\; t \in [0, \infty). 
\eeq
In view of the inequalities \eqref{H1b} and \eqref{ut0} adapted to the integral over time interval $[\frac{1}{2}, 1]$,  we can assert that for any given $g_0 \in \mathscr{A}$ there is a time point $t_0 \in [\frac{1}{2}, 1]$, which may depend on $g_0$, such that the $u$- component $u(t)$ of the solution $S(t)g_0$ satisfies
\beq \bl{uC}
	\|u(t_0)\|^2_{C(\overline{\gw})} \leq C_{emb} \|u(t_0)\|_{H^1}^2 \leq \frac{C_{emb}}{C_1 \min \{\eta, \gl \}} \left(\max \{C_1, 1\}K + \gl M | \gw|\right),
\eeq
where $C_{emb}$ is a Sobolev embedding constant for $H^1 (\gw) \hookrightarrow C(\overline{\gw})$, under the condition $\dim \, \gw = 1$. Consequently, due to the compactness of $\ms{A}$, the compactness  of the time interval $[\frac{1}{2}, 1]$, and the strong continuity of $u(t)$  in the space $C(\overline{\gw})$ with respect to $t$, the inequality \eqref{uC} infers that there exists a finite positive constant 
\beq \bl{R}
	R\, (\text{being fixed}) \, \geq \frac{C_{emb}}{C_1 \min \{\eta, \gl \}} \left(\max \{C_1, 1\}K + \gl M | \gw|\right)
\eeq 
such that
$$
	\sup_{g_0 \in \mathscr{A}} \|u(1)\|^2_{L^{\infty}} = \sup_{g_0 \in \mathscr{A}} \|u(1)\|^2_{C(\overline{\gw})} \leq R, 
$$
because the invariance of $\mathscr{A}$ in \eqref{Sta} tells us $S(1) \ms{A} = \ms{A}$. Therefore, the $u$-projection of the global attractor $\mathscr{A}$ is a bounded subset inside the ball 
$$
	B_{C} = \{u \in C(\overline{\gw}): \|u\|^2_{C(\overline{\gw})} \leq R\}
$$
in the space $C(\overline{\gw})$ and in $L^{\infty} (\gw)$ as well.
\end{proof}

\begin{lemma} \bl{vwrh}
	The $(v, w, \rho)$-projection of the global attractor $\mathscr{A}$ is a bounded set in the space $L^\infty (\gw, \mathbb{R}^3)$.
\end{lemma}

\begin{proof}
	The integral version of the differential equations \eqref{vcw}, \eqref{wq} and \eqref{rq} shows that, for any $g_o = (u_0, v_0, w_0, \rho_0) \in \mathscr{A}$, one has 
\beq \bl{vwr}
	\begin{split}
	&\lim_{t \to \infty} \left| v(t, x) - \int_{0}^t e^{-(t-s)} (\alpha - \beta u^2)\, ds \right| \leq \lim_{t \to \infty} e^{-t} |v_0(x)| = 0 ,  \\
	&\lim_{t \to \infty} \left| w(t, x) - \int_{0}^t e^{-r (t-s)} q (u - u_e)\, ds \right| \leq \lim_{t \to \infty}e^{- rt} |w_0(x)| = 0,   \\
	&\lim_{t \to \infty} \left|\rho (t, x) - \int_{0}^t e^{-k_2 (t-s)} u(s,x)\, ds \right| \leq \lim_{t \to \infty} e^{- k_2 t} |\rho_0(x)| = 0.
	\end{split}
\eeq
By Lemma \ref{Lcb}, we see that
\begin{equation*}
	\max \left\{ \left| \int_{t_0}^t e^{-(t-s)} (\alpha - \beta u^2) ds \right|, \left| \int_{t_0}^t e^{-r (t-s)} q (u - u_e) ds \right|, \left| \int_{t_0}^t e^{-k_2 (t-s)} u(s,x) ds \right| \right\} \leq G,
\end{equation*}
where 
\beq \bl{G}
	G = \max \left\{| \ap - \gb R |, \,\, \frac{q}{r} \left(\sqrt{R} + |u_e| \right), \, \frac{1}{k_2} \sqrt{R} \right\}.
\eeq
The compactness and invariance \eqref{Sta} of the global attractor $\ms{A}$ then implies that 
\begin{equation*}
	\lim_{t \to \infty} \dist_{\mathbb{R}^3} \left(\text{Proj}_{(v,w,\rho)} S(t)\mathscr{A}, B_{R^3}(G)\right) = \dist_{\mathbb{R}^3} \left(\text{Proj}_{(v,w,\rho)} \mathscr{A}, B_{R^3} (G)\right) = 0.
\end{equation*}
Here $B_{\mathbb{R}^3}(G)$ is the 3D bounded ball of radius $G$. It means that $\text{Proj}_{(v,w,\rho)} \mathscr{A}$ is a bounded set in the space $L^\infty (\gw, \mathbb{R}^3)$ and the Lemma is proved. 
\end{proof}

\begin{theorem}\label{reg}
	For spatial domain dimension $\dim\, (\gw) = 1$, the global attractor $\mathscr{A}$ in \eqref{GA} of the memristive Hindmarsh-Rose semiflow $\{S(t)\}_{t \geq 0}$ is a bounded set in the regular space $\Gamma = H^2(\gw) \times L^\infty (\gw, \mathbb{R}^3)$.
\end{theorem}

\begin{proof} 
Consider all the solution trajectories $\{S(t)g_0: g_0 \in \mathscr{A}\}$, which are complete trajectories in terms of $t \in (-\infty, \infty)$ and all inside the global attractor $\mathscr{A}$. In view of Lemma \ref{vwrh}, it suffices to prove that the $u$-projection of the global attractor $\ms{A}$ is in the space $H^2 (\gw)$. The proof goes through three steps.

Step 1. For the first component $u(t, x)$ of any solution trajectory in $\mathscr{A}$, take the $L^2$ inner-product $\inpt{\eqref{ueq}, u_t}$, where $u_t = \frac{\partial u}{\partial t}$, to obtain
\begin{equation} \bl{ut}
	\begin{split}
	&\|u_t\|^2 + \frac{\eta}{2}\frac{\partial}{\partial t}\|\nb u\|^2 = \int_\gw (a u^2 - b u^3 + v - w + J_e - k_1 (c + \ga \rho + \de \rho^2) u_t)\, dx\\
	\leq &\, \int_\gw \left(a R + b R^{3/2} + | v | + | w | + J_e + k_1 | c | + k_1 \ga^2 + k_1 (1 + \de)\rho^2 \right) |u_t|\, dx \\[2pt]
	\leq &\, \left(a R + b R^{3/2} + J_e + k_1 | c | + k_1 \ga^2 \right)^2 |\gw|   \\[5pt]
	&+ \| v\|^2 + \| w\|^2 + k_1^2 (1 + \de)^2 \|\rho^2\|^2 + \frac{3}{4} \|u_t\|^2  \\[5pt]
	\leq &\, \left(a R + b R^{3/2} + J_e + k_1 | c | + k_1 \ga^2 \right)^2 |\gw| + K + k_1^2 (1 + \de)^2 \|\rho\|^4_{L^4} + \frac{3}{4} \|u_t\|^2  \\[4pt]
	\leq &\, \left(a R + b R^{3/2} + J_e + k_1 | c | + k_1 \ga^2 \right)^2 |\gw| + K + k_1^2 (1 + \de)^2 G^4 \,|\gw| + \frac{3}{4} \|u_t\|^2.
	\end{split}	
\end{equation}
The Young's inequality \eqref{Yg} is used in the second inequality of \eqref{ut}.

Take the $L^2$ inner-product $\inpt{\eqref{veq}, v_t}$ for the second component $v(t, x)$ of any trajectory in $\mathscr{A}$. We have
\begin{equation} \bl{vt}
	\begin{split}
	&\|v_t\|^2 = \int_\gw (\alpha - \beta u^2 - v) v_t\, dx\\
	\leq &\, \int_\gw (\alpha + \beta R + | v |) |v_t|\, dx = (\alpha + \beta R)^2\, |\gw | + K + \frac{1}{2}\|v_t\|^2 
	\end{split}
\end{equation}
where the Cauchy inequality and \eqref{lsp2} are used. Similarly take the $L^2$ inner-product $\inpt{\eqref{weq}, w_t}$ for the third component $w(t, x)$ of any trajectory in $\mathscr{A}$. We see 
\begin{equation} \bl{wt}
	\begin{split}
	&\|w_t\|^2 = \int_\gw (q u - q u_e - r w) w_t\, dx.     \\
	\leq &\, \int_\gw (q \sqrt{R} + q |u_e| + r | w |)|w_t|\, dx = (q \sqrt{R} + q |u_e|)^2 \, |\gw | + r^2 K + \frac{1}{2} \|w_t\|^2
	\end{split}
\end{equation}
Next take the $L^2$ inner-product $\inpt{\eqref{peq}, \rho_t}$ for the fourth component $\rho (t, x)$ of any trajectory in $\mathscr{A}$ to get
\beq \bl{rhot}
	\|\rho_t\|^2 = \int_\gw (u - k_2 \rho) \rho_t \, dx \leq R\,|\gw | + k_2^2 K + \frac{1}{2} \|\rho_t \|^2.
\eeq

Now Sum up the above four estimate inequalities. We come up with
\beq \label{p3}
	\begin{split}
	&\frac{1}{4} \left(\|u_t\|^2 + \|v_t\|^2 + \|w_t\|^2 + \|\rho_t \|^2\right)+ \frac{\eta}{2} \frac{\partial}{\partial t} \|\nb u\|^2  \\[3pt]
	\leq &\frac{1}{4}\|u_t\|^2 + \frac{1}{2} \left(\|v_t\|^2 + \|w_t\|^2 + \|\rho_t \|^2\right)+ \frac{\eta}{2} \frac{\partial}{\partial t} \|\nb u\|^2  \\[3pt]
	\leq &\, \left(a R + b R^{3/2} + J_e + k_1 | c | + k_1 \ga^2 \right)^2 |\gw| + K + k_1^2 (1 + \de)^2 G^4 \,|\gw|   \\[8pt]
	+ &\, (\alpha + \beta R)^2\, |\gw| + K + (q \sqrt{R} + q |u_e|)^2\, |\gw | + r^2 K + R\,|\gw | + k_2^2 K,   \\[6pt]
	= &\, \left(\left[a R + b R^{3/2} + J_e + k_1 | c | + k_1 \ga^2 \right]^2 + k_1^2 (1 + \de)^2 G^4 \right)|\gw|     \\[2pt]
	+ &\, \left(2 + r^2 + k_2^2 \right) K +  \left[(\alpha + \beta R)^2 + (q \sqrt{R} + q |u_e|)^2 + R \right] |\gw |,   \quad t > 0.
	\end{split}
\eeq
Integrating the inequality \eqref{p3} over the time interval $[0,1]$, we obtain 
\beq \label{p4}
	\int_{0}^{1} \left(\|u_t (s)\|^2 + \|v_t (s)\|^2 + \|w_t (s)\|^2 + \|\rho_t (s)\|^2\right) ds \leq \Phi ,
\eeq
where the constant 
\begin{equation}. \bl{Phi}
	\begin{split}
	\Phi &= \frac{2\eta}{C_1 \min \{\eta, \gl \}} \left(\, \max \{C_1, 1\} K + \gl M | \gw| \,\right)   \\
	&+ 4\left(\left[a R + b R^{3/2} + J_e + k_1 | c | + k_1 \ga^2 \right]^2 + k_1^2 (1 + \de)^2 G^4 \right)|\gw|     \\
	&+ 4\left(2 + r^2 + k_2^2 \right) K + 4 \left[(\alpha + \beta R)^2 + (q \sqrt{R} + q |u_e|)^2 + R \right] |\gw |.
	\end{split}
\end{equation}	
and in the inequality \eqref{p4} we have used \eqref{ut0} and \eqref{Sta} to bound the gradient term $(\eta/2) \|\nb u_0 \|^2$ for all the trajectories in the global attractor.

Step 2. For the diffusive Hindmarsh-Rose equations with memristors confined on the set of the global attractor $\ms{A}$, we differentiate the equation \eqref{ueq} with respect to time $t$ and get
\begin{equation} \bl{Bp}
	u_{tt} = \eta \gd u_t +  2 a u\, u_t - 3\, b u^2 u_t + v_t - w_t - k_1 \vp (\rho) u_t - k_1 (\ga \rho_t + 2\de \rho \rho_t)u.
\end{equation}
Take the inner product $\inpt{\eqref{Bp}, t^2 u_t}$. Based on the results shown in Lemma \ref{Lcb} and Lemma \ref{vwrh}, it yields 

\begin{equation} \label{p8} 
	\begin{split}
	& \inpt{u_{tt}, t^2 u_t} - \langle \eta \Delta u_t, t^2 u_t \rangle = -t\|u_t\|^2 + \frac{1}{2} \frac{d}{dt} \|t u_t\|^2 + t^2 \eta \|\nb u_t\|^2    \\
	= &\, \int_\gw t^2 (2 a u\, u_t^2 - 3 b u^2 u_t^2 +v_t u_t - w_t u_t - k_1 \vp (\rho) u^2_t - k_1 (\ga \rho_t + 2\de \rho \rho_t)u u_t)\, dx\\
	\leq &\, t^2 \int_\gw \left[2a \sqrt{R} \, u_t^2 + v_t^2 + w_t^2 + 2u_t^2 + k_1 |\vp(G)| u^2_t + k_1 \sqrt{R}\, (|\ga| + \de G)\, (\rho_t^2 + u_t^2) \right] dx \\
	= &\, t^2 \left[2a \sqrt{R} + 2 + k_1 (|\vp(G)| + \sqrt{R}\, (|\ga| + \de G))\right] \|u_t\|^2 + t^2 \left(\|v_t\|^2 + \|w_t\|^2\right) \\[4pt]
	&+ t^2 k_1 \sqrt{R} \, (|\ga| + \de G) \, \| \rho_t \|^2 , \quad t > 0.
	\end{split}
\end{equation}
Here $\vp (G) = c + \ga G + \de G^2$ and the first two terms in the first equality of \eqref{p8} is derived by
\begin{equation*}
	\begin{split}
	&\quad -t \|u_t\|^2 + \frac{1}{2} \frac{d}{dt} \|t u_t\|^2  = -t \|u_t\|^2 + \frac{1}{2} \frac{d}{dt} \inpt{t u_t, t u_t}  \\[3pt]
	& = -t \|u_t\|^2 + \frac{1}{2} \left(\left\langle \frac{\partial}{\partial t} (t u_t), t u_t \right\rangle + \left\langle t u_t, \frac{\partial}{\partial t} (t u_t) \right\rangle \right) \\
	& = -t \|u_t\|^2 + \left\langle \frac{\partial}{\partial t} (t u_t), t u_t \right\rangle = -t \|u_t\|^2 + \inpt{u_t, t u_t} + \inpt{t u_{tt}, t u_t}.  \\[3pt]
	& = -t \|u_t\|^2 + t \|u_t\|^2 + \inpt{u_{tt}, t^2 u_t} = \inpt{u_{tt}, t^2 u_t}.
	\end{split}
\end{equation*}
Now we integrate the differential inequality \eqref{p8} on $[0, t]$ to obtain 
\begin{equation*}
	\begin{split}
	\frac{1}{2} \|t u_t\|^2 &\leq  \int_{0}^{t} s \|u_t(s)\|^2 ds + \int_{0}^{t} s^2 \left(\|v_t\|^2\ + \|w_t\|^2 + k_1 \sqrt{R}\,(|\ga| + \de G)\, \| \rho_t \|^2\right) ds   \\
	&+ \int_{0}^{t} s^2 \left( 2a \sqrt{R} + 2 + k_1 (|\vp(G)| + \sqrt{R}(|\ga| + \de G)) \right) \|u_t(s)\|^2\, ds,  \;\;  t > 0.
	\end{split}
\end{equation*}
In the above inequality we can take $t = 1$ and get
\begin{equation} \label{p9}
	\begin{split}
	\| u_t (1)\|^2 &\leq 2 \int_{0}^{1} \|u_t(s)\|^2\, ds + 2\int_{0}^{1} \left(\|v_t \|^2 + \|w_t \|^2 + k_1 \sqrt{R}\,(|\ga| + \de G)\, \| \rho_t \|^2\right) ds   \\
	&+ 2 \int_{0}^{1} \left( 2a \sqrt{R} + 2 + k_1 (|\vp(G)| + \sqrt{R}\,(|\ga| + \de G)) \right) \|u_t(s)\|^2\, ds \leq D.
	\end{split}
\end{equation}  
By the inequality in \eqref{p4}, here we have the constant
\begin{equation} \bl{D}
	D = 2 \left(3 + 2a \sqrt{R} + k_1 |\vp(G)| + 2k_1 \sqrt{R}\, (|\ga| + \de G)\right) \Phi ,
\end{equation}
where the constant $\Phi$ is given in \eqref{Phi} of Step 1. Due to the dynamic invariance of the global attractor $\ms{A}$,
$$
	S(t + 1) \mathscr{A} = S(t) \mathscr{A} = \mathscr{A}, \quad \text{for any} \;\, t \in [0, \infty), 
$$
actually \eqref{p9} demonstrates that for all the trajectories $\{S(t) g_0\}_{t \geq 0}$ in the global attractor $\ms{A}$, the $u$-component satisfies
\beq \bl{utbd}
	\| u_t (t)\|^2 \leq D,  \quad \text{for any} \;\, t \in [0, \infty). 
\eeq
	
	Step 3. From the original $u$-equation \eqref{ueq}, put together what we have proved in above steps, it holds that
\begin{equation} \label{p10}
	\begin{split}
	&\eta \|\gd u(t)\| \leq \|u_t (t)\| + a \|u^2 (t)\| + b\|u^3 (t) \| + \|v(t)\| + \|w(t)\| + J_e |\gw|^{1/2}   \\[4pt] 
	&+ k_1 \|(c + \ga \rho (t) + \de \rho^2 (t))u(t)\|   \\
        \leq &\, \sqrt{D} + \left(a R + b R^{3/2} + 2G + J_e + k_1 (|c| + |\ga | G + \de G^2)\sqrt{R}\, \right) |\gw |^{1/2}, \;\;  t > 0.
        	\end{split}
\end{equation}
and the constants $R$ in \eqref{R}, $G$ in \eqref{G}, and $D$ in \eqref{D} are all independent of any initial state in the space $E$.
	
	Since the Laplacian operator $\Delta$ with the homogeneous Neumann boundary condition \eqref{nbc} is self-adjoint and negative definite modulo constant functions, the Sobolev space norm of any function $h(x)$ in $H^2 (\gw)$ is equivalent to $\| h \| + \eta \|\gd h \|$. Therefore, the inequality \eqref{p10} together with Theorem \ref{T2} and the fact $S(t) \mathscr{A} = \mathscr{A}$ shows that the $u$-component of the global attractor $\mathscr{A}$ is a bounded set in $H^2 (\gw)$.  A quantitative bound of the equivalent $H^2$-norm is given by
\beq \bl{H2}
	\sup_{g \,\in \,\mathscr{A}} \|u\|_{H^2} \leq \sqrt{K} + \sqrt{D} + \left(a R + b R^{3/2} + 2G + J_e + k_1 (|c| + |\ga | G + \de G^2)\sqrt{R}\right) |\gw |^{1/2}.
\eeq
Combined \eqref{H2} with Lemma \ref{Lcb} and Lemma \ref{vwrh}, we have proved that the global attractor $\ms{A}$ is a bounded set in the regular space $\Gamma = H^2 (\gw) \times L^\infty (\gw, \mathbb{R}^3)$. 
\end{proof}

\textbf{Conclusions}. In this paper, the diffusive Hindmarsh-Rose equations with memristors are proposed as a new mathematical model of neuron dynamics, which is a hybrid system of coupled partial differential equation of the membrane potential for a neuron cell and three ordinary differential equations of the fast and slow ion channels plus a memristive variable featuring the dynamical memory due to the electromagnetic flux effect. The rationality of such a model is at least biologically explicit in view of the long axons of neuron cells in brain and nerve systems. 

Global dynamics for the solution semiflow of this memristive Hindmarsh-Rose system is studied under no conditions of any naturally involved biological and mathematical parameters. The main result is the existence of a unique global attractor for this dynamical system or called semiflow generated by the weak solutions in the basic state space $E = L^2 (\gw, \mathbb{R}^4)$. 

Due to the quadratic nonlinear memductance and its nonlinear coupling with the membrane potential variable in the main partial differential equation, the challenging proofs of dissipativity and asymptotic compactness are carried out through many steps of sophisticated \emph{a priori} uniform estimates and the Kolmogorov-Riesz compactness approach. 

Moreover, the spatial regularity of this global attractor in the space $\Gamma = H^2 (\gw) \times L^\infty (\gw, \mathbb{R}^3)$ is also proved for one-dimensional domain. The quantitative bounds of the region containing this global attractor in the state space $E$ and the region in the regular space $\Gamma$ are explicitly provided, which can be used to facilitate further researches on stable or unstable equilibria patterns, coexisting chimeras, bifurcation and firing patterns, or chaotic local attractors. All the permanent regimes of the modeled neuron dynamics must be included in the global attractor and located in these two regions. 

\bibliographystyle{amsplain}

\begin{thebibliography}{99}

\bibitem{Ay}
I. K. Aybar, \emph{Memristor-based oscillatory behavior in the FitzHugh-Nagumo and Hindmarsh-Rose models}, Nonlinear Dynamics, \textbf{103} (2021), 2917-2929.

\bibitem{BKG}
Y. Babacan, F. Kacar and K. Gurkan, \emph{A spiking and bursting neuron circuit based on memristor}, Neurocomputing, \textbf{203} (2016), 86-91.

\bibitem{BB} 
B. Bao \emph{et al}, \emph{Three-dimensional memristive Hindmarsh-Rose neuron model with hidden coexisting asymmetric behaviors}, Complexity, \textbf{2018} (2018), 3872573.

\bibitem{BH}
H. Bao \emph{et al}, \emph{Hidden bursting firings and bifurcation mechanisms on memristive neuron model with threshold electromagnetic induction}, IEEE Transactions on Neural Networks and Learning Systems, \textbf{31}(2) (2020), 502-511. 

\bibitem{BRS}
R.J. Buters, J. Rinzel and J.C. Smith, \emph{Models respiratory rhythm generation in the pre-B\"{o}tzinger complex, I. Bursting pacemaker neurons}, J. Neurophysiology, \textbf{81} (1999), 382--397.

\bibitem{CK}
T.R. Chay and J. Keizer, \emph{Minimal model for membrane oscillations in the pancreatic beta-cell}, Biophysiology Journal, \textbf{42} (1983), 181--189.

\bibitem{CV}
V.V. Chepyzhov and M. I. Vishik, \emph{Attractors for Equations of Mathematical Physics},  AMS Colloquium Publications, Vol. \textbf{49}, AMS, Providence, RI, 2002.

\bibitem{Chua}
L. Chua, \emph{Memristor - the missing circuit element}, IEEE Trans. Circuit Theory, \textbf{18} (1971), 507.

\bibitem{ChuaK}
L. Chua and S.M. Kang, \emph{Memristive devices and systems}, Proceedings of the IEEE, \textbf{64}(2) (1976), 209-223.

\bibitem{CS}
L.N. Cornelisse, W.J. Scheenen, W.J. Koopman, E.W. Roubos and S.C. Gielen, \emph{Minimal model for intracellular calcium oscillations and electrical bursting in melanotrope cells of Xenopus Laevis}, Neural Computations, \textbf{13} (2000), 113--137.

\bibitem{DR}
I.S. Doubla \emph{et al}, \emph{Infinitely many coexisting hidden attractors in a new hyperbolic-type memristor-based HNN}, European Physical Journal Special Topics, (2022), 1-15.

\bibitem{ET}
G.B. Ementrout and D.H. Terman, \emph{Mathematical Foundations of Neurosciences}, Springer, 2010. 

\bibitem{EE}
A.S. Et\'{e}m\'{e} \emph{et al}, \emph{Chaos break and synchrony enrichment within Hindmarsh-Rose-type memristive neural models}, Nonlinear Dynamics, \textbf{105} (2021), 785-795.

\bibitem{FH}
R. FitzHugh, \emph{Impulses and physiological states in theoretical models of nerve membrane}, Biophysical Journal, \textbf{1} (1961), 445--466.

\bibitem{Guan}
W. Guan, S. Yi and Y. Quan, \emph{Exponential synchronization of coupled memristive neural networks via pinning control}, Chinese Physics B, \textbf{22} (2013), 050504.

\bibitem{Han}
B. Han, W. Liu and A. Hu, \emph{Coexisting multiple firing patterns in two adjacent neurons coupled by memristive electromagnetic induction}, Nonlinear Dynamics, \textbf{95} (2019), 43-56.

\bibitem{HO}
H. Hanche-Olsen and H. Holden, \emph{The Kolmogorov-Riesz compactness theorem}, Expositiones Mathematicae, \textbf{28} (2010), 385--394.

\bibitem{HR}
J.L. Hindmarsh and R.M. Rose, \emph{A model of neuronal bursting using three coupled first-order differential equations}, Proceedings of the Royal Society London, Series B: Biological Sciences,  \textbf{221} (1984), 87--102.

\bibitem{HH}
A. Hodgkin and A. Huxley, \emph{A quantitative description of membrane current and its application to conduction and excitation in nerve}, J. Physiology, Series B,  \textbf{117} (1952), 500--544.

\bibitem{HY}
M. Hui and J. Yan, \emph{Integral sliding mode exponential synchronization of inertial memristive neural networks with time varying delays}, Neural Processing Letters, (2022). https://doi.org/ 10.1007/s11063-022-10981-9.

\bibitem{IG}
G. Innocenti and R. Genesio, \emph{On the dynamics of chaotic spiking-bursting transition in the Hindmarsh-Rose neuron}, Chaos, \textbf{19} (2009), 023124.

\bibitem{Ng}
E.B.M. Ngouonkadi \emph{et al}, \emph{Bifurcations and multistability in the extended Hindmarsh-Rose neuronal oscillator}, Chaos, Solitons and Fractals, \textbf{85}(4) (2016), 151-163.

\bibitem{PYS1}
C. Phan, Y. You and J. Su, \emph{Global attractor for Hindmarsh-Rose equations in neurodynamics}, Journal of Nonlinear Modeling and Analysis, \textbf{2}(4) (2020), 559-577.

\bibitem{PY}
C. Phan and Y. You, \emph{Exponential attractor for Hindmarsh-Rose equations in neurodynamics}, Journal of Applied Analysis and Computation, \textbf{10}(5) (2020), 1-22.

\bibitem{PYS2}
C. Phan, Y. You and J. Su, \emph{Global dynamics of partly diffusive Hindmarsh-Rose equations in neurodynamics}, Dynamics of Partial Differential Equations, \textbf{18}(1) (2021), 33-47.

\bibitem{CY}
C. Phan and Y. You, \emph{Random attractor for stochastic Hindmarsh-Rose equations with additive noise}, Journal of Dynamics and Differential Equations, \textbf{33} (2021), 489-510.

\bibitem{CPY}
C. Phan and Y. You, \emph{Synchronization of boundary coupled Hindmarsh-Rose neuron network}, Nonlinear Analysis: Real World Applications, \textbf{55} (2020), 103139.

\bibitem{PSY}
C. Phan, L. Skrzypek and Y. You, \emph{Dynamics and synchronization of complex neural networks with boundary coupling}, Analysis and Mathematical Physics, (2022) 12:33. 

\bibitem{QW}
G. Qi and Z. Wang, \emph{Modeling and dynamics of double Hindmarsh-Rose neuron with memristor-based magnetic coupling and time delay}, Chinese Physics B, \textbf{30}(12) (2021), 120516.

\bibitem{RJ}
K. Rajagopal \emph{et al}, \emph{Effect of magnetic induction on the synchronizability of coupled neuron network},, Chaos, \textbf{31} (2021), 083115.

\bibitem{RM}
B. Ramakrishnan \emph{et al}, \emph{A new memristive neuron map model and its network's dynamics under electrochemical coupling}, Electronics, \textbf{11} (2022), 153. 

\bibitem{SY} 
G.R. Sell and Y. You, \emph{Dynamics of Evolutionary Equations}, Applied Mathematical Sciences, Volume \textbf{143}, Springer, New York, 2002.

\bibitem{SC}
A. Shapiro, R. Curtu, J. Rinzel and N. Rubin, \emph{Dynamical characteristics common to neuronal competition models}, J. Neurophysiology, \textbf{97} (2007), 462--473.

\bibitem{SWZ}
X. Shi, Z. Wang and L. Zhuang, \emph{Spatiotemporal pattern in a neural network with non-smooth memristor}, Electronic Research Archive, \textbf{30}(2) (2022), 715-731. 

\bibitem{SR}
P.P. Singh, A. Rai and B.K. Roy, \emph{Memristor-based asymmetric extreme multistate hyperchaotic system with a line of equilibria, coexisting attractors, its implementation and nonlinear active-adaptive projective synchronization}, European Physical Journal Plus, (2022) 137:875. https://doi.org/10.1140/epjp/s13360-022-03063-1.

\bibitem{SS}
D.B. Strukov \emph{et al}, \emph{The missing memristor found}, Nature, \textbf{453} (2008), 80.

\bibitem{SPH}
J. Su, H. Perez-Gonzalez and M. He, \emph{Regular bursting emerging from coupled chaotic neurons}, Discrete and Continuous Dynamical Systems, Supplement 2007, 946--955.

\bibitem{US1}
K. Usha and P.A. Subha, \emph{Hindmarsh-Rose neuron model with memristors}, BioSystems, \textbf{178} (2019), 1-9. https://doi.org/10.1016/j.biosystems.2019.01.005.

\bibitem{US2}
K. Usha and P.A. Subha, \emph{Energy feedback and synchronous dynamics of Hindmarsh-Rose neuron model with memristor}, Chinese Physics B, \textbf{28}(2) (2019), 020502.

\bibitem{VK}
C.K. Volos \emph{et al}, \emph{Memristor: A new concept in synchronization of coupled neuromorphic circuits}, J. Eng. Sci. Tech. Review, \textbf{8} (2015), 157.

\bibitem{Wang}
Y. Wang, \emph{An image encryption scheme by applying memristive Hindmarsh-Rose neuron model}, Physica Scripta, \textbf{97} (2022), 075202.

\bibitem{Wu}
F. Wu, H. Gu and Y. Li, \emph{Inhibitory electromagnetic induction current induces enhancement instead of reduction of neural bursting activities}, Communications in Nonlinear Science and Numerical Simulation, \textbf{79} (2019), 104924.

\bibitem{XJ}
Y. Xu \emph{et al}, \emph{Synchronization between neurons coupled by memristor}, Chaos, Solitons and Fractals, \textbf{104} (2017), 435.

\bibitem{XQM}
L. Xu, G. Qi and J. Ma, \emph{Modeling of memristor-based Hindmarsh-Rose neuron and its dynamical analysis using energy method}, Applied Mathematical Modeling, \textbf{101} (2022), 503-516.

\bibitem{YP}
B. Yan \emph{et al}, \emph{Further dynamical analysis of modified FitzHUgh-Nagumo model under the electric field}, Nonlinear Dynamics, \textbf{101}(1) (2020), 521-529. 

\end{thebibliography}

\end{document}